\documentclass[a4paper,12pt]{amsart}

\usepackage[margin=1in]{geometry}

\usepackage{amsmath, amsthm, amssymb}
\usepackage{cite}
\usepackage[dvipdfmx]{graphicx}
\usepackage{bm}
\usepackage{datetime2} 

\theoremstyle{definition}
\newtheorem{theorem}{Theorem}[section]
\newtheorem{lemma}[theorem]{Lemma}
\newtheorem{proposition}[theorem]{Proposition}
 
\newtheorem{observation}[theorem]{Observation} 
\newtheorem{corollary}[theorem]{Corollary} 

\theoremstyle{definition}
\newtheorem{definition}[theorem]{Definition}

\newtheorem{pclaim}[theorem]{Claim}
\newtheorem*{ac}{Acknowledgments} 

\theoremstyle{remark}
\newtheorem{remark}[theorem]{Remark}

\newenvironment{rmenum}{
\begin{enumerate}

}
{\end{enumerate}}

\newcommand{\parNei}[2]{N_{#1}(#2)}
\newcommand{\shore}[2]{\partial_{#1}(#2)}

\newcommand{\tcomp}[2]{\mathcal{G}(#1, #2)}

\newcommand{\distgt}[4]{\lambda_{(#1, #2)}(#3, #4)}

\newcommand{\distgtf}[5]{\lambda_{(#1, #2)}(#4, #5; #3)}
\newcommand{\parcut}[2]{\delta_{#1}(#2)}

\newcommand{\gtsim}[2]{\sim_{(#1, #2)}}
\newcommand{\tpart}[2]{\mathcal{P}(#1, #2)}

\newcommand{\pargtpart}[3]{\mathcal{P}(#1, #2)|_{#3}}

\newcommand{\laylegtr}[4]{L_{(#1, #2)}(#4; #3)}

\newcommand{\levelgtr}[4]{U_{(#1, #2)}(#4; #3)}

\newcommand{\laycompgtr}[4]{\mathcal{L}_{(#1, #2)}(#4; #3)}
\newcommand{\laycompall}[3]{\mathcal{L}_{(#1, #2)}(#3)}

\newcommand{\agtr}[3]{A_{(#1, #2)}(#3)}

\newcommand{\interval}[3]{I_{(#1, #2)}(#3)}

\newcommand{\ak}[3]{{A}_{(#1, #2)}(#3)} 
\newcommand{\akfam}[3]{{\mathcal{A}}_{(#1, #2)}(#3)} 
\newcommand{\dk}[3]{{D}_{(#1, #2)}(#3)} 
\newcommand{\dkconn}[3]{{\mathcal{D}}_{(#1, #2)}(#3)} 
\newcommand{\neicompk}[4]{\mathcal{B}_{(#1, #2)}(#4; #3)} 
\newcommand{\neisetk}[4]{B_{(#1, #2)}(#4; #3)} 
\newcommand{\noncapgtr}[4]{\mathcal{L}^{*}_{#4}( #3; #1, #2)}
\newcommand{\noncapall}[3]{\mathcal{L}^{*}_{(#1, #2)}(#3)}
\newcommand{\negset}[3]{D^{\circ}_{(#1, #2)}(#3)} 
\newcommand{\negsetnf}[5]{D^{\circ}_{(#1, #2),  #3}(#4 \mid #5)} 
\newcommand{\negsetn}[4]{D^{\circ}_{(#1, #2)}(#3 \mid #4)}

\title[Odd Cuts in Bipartite Grafts II]{Odd Cuts in Bipartite Grafts II: Structure and Universality of Decapital Distance Components}
\author{Nanao Kita}
\address{Nagoya University, 464-8602 Furocho, Chikusa, Nagoya, Japan}
\email{kita@math.nagoya-u.ac.jp}

\date{\today}

\begin{document}

\begin{abstract} 
This paper is the second in a series of papers characterizing the maximum packing of \( T \)-cuts in bipartite grafts, following the first paper (N.~Kita, ``Tight cuts in bipartite grafts~I: Capital distance components,'' {arXiv:2202.00192v2}, 2022).  
Given a graft $(G, T)$, a minimum join $F$, and a specified vertex $r$ called the root, 
the distance components of $(G, T)$ are defined as subgraphs of $G$ determined by the distances induced by $F$.  
A distance component is called {\em capital} if it contains the root; otherwise, it is called {\em decapital}. 
In our first paper, we investigated the canonical structure of capital distance components in bipartite grafts, 
which can be described using the graft analogue of the Kotzig--Lov\'asz decomposition.
In this paper, 
we provide the counterpart structure for the decapital distance components. 
We also establish a necessary and sufficient condition for two vertices $r$ and $r'$  
under which a decapital distance component with respect to root $r$ 
is also a decapital distance component with respect to root $r'$. 
As a consequence, we obtain that the total number of decapital distance components in a bipartite graft, taken over all choices of root, is equal to twice the number of edges in a minimum join of the graft. 
\end{abstract}

\maketitle

\section{Grafts and Joins}

For basic definitions and notation on sets and graphs, we mostly follow Schrijver~\cite{schrijver2003}. 
We also adopt definitions and notation from our previous paper~\cite{kita2022tight}, particularly for nonstandard or specialized usage. 
In this section, we provide fundamental definitions and results on grafts and joins.

\begin{definition} 
A pair of a graph $G$ and a set $T\subseteq V(G)$ is a {\em graft} 
if $|V(C)\cap T|$ is even for every connected component $C$ of $G$.  
For a graft $(G, T)$, a set $F\subseteq E(G)$ is a {\em join} of $(G, T)$ if $|\parcut{G}{v} \cap T|$ is odd for every $v\in T$ and is even for every $v\in V(G)\setminus T$. 
\end{definition} 

\begin{observation}[See Schrijver~\cite{schrijver2003}]  \label{obs:join} 
Every graft has a join. 
\end{observation} 

\begin{definition} 
Under Observation~\ref{obs:join}, we denote the size of a minimum join of a graft $(G, T)$, that is, the number of edges in a join with the minimum number of edges, by $\nu(G, T)$. 
\end{definition}

\begin{definition} 
Let $(G, T)$ be a graft, and let $F\subseteq E(G)$ be a join of $(G, T)$.  
Let $X \subseteq V(G)$. 
We denote by $(G, T)_F[X]$ the graft $(G[X], T')$ 
where a vertex $v\in X$ is an element of $T'$ if and only if 
$|\parcut{G}{v}\cap (F\cap E[X])|$ is odd. 
\end{definition}

\section{Distances in Grafts and Seb\H{o}'s Distance Theorem}

In this section, we present concepts and known results on distances in grafts determined by a minimum join, which were introduced in Seb\H{o}~\cite{sebo1990}. 
We also introduce some notation related to Seb\H{o}'s results to be used in later sections.

\begin{definition} 
Let $(G, T)$ be a graft, and let $F\subseteq E(G)$. 
We define a mapping $w_F: E(G) \rightarrow \{-1, +1\}$ as 
$w_F(e) = 1$ if $e\in E(G)\setminus F$ holds and $w_F(e) = -1$ otherwise. 
For a set $S \subseteq E(G)$, we denote $\sum_{e\in S} w_F(e)$ as $w_F(S)$. 
For a subgraph $H$ of $G$, $w_F(H)$ denotes $w_F(E(H))$. 
We call $w_F(e)$, $w_F(S)$, or $w_F(H)$ the {\em $F$-weight} of $e$, $S$, or $H$, respectively. 
For $x, y \in V(G)$, $\distgtf{G}{T}{F}{x}{y}$ denotes the minimum of $w_F(P)$, 
where $P$ is taken over every path in $G$ between $x$ and $y$. 
We call $\distgtf{G}{T}{F}{x}{y}$ the {\em $F$-distance} between $x$ and $y$. 
\end{definition}

\begin{proposition}[Seb\H{o}~\cite{sebo1990}] \label{prop:distcanonical} 
Let $(G, T)$ be a graft, and let $F$ be a minimum join. 
For every $x, y\in V(G)$, 
$\distgtf{G}{T}{F}{x}{y}$ is equal to $\nu(G, T) - \nu(G, T\Delta \{x, y\})$. 
\end{proposition}

\begin{definition} 
Under Proposition~\ref{prop:distcanonical}, 
for a graft $(G, T)$ and vertices $x, y\in V(G)$, 
we denote by $\distgt{G}{T}{x}{y}$ the value $\nu(G, T) - \nu(G, T\Delta \{x, y\})$, 
which is equal to $\distgtf{G}{T}{F}{x}{y}$ for every minimum join $F$. 
\end{definition}

\begin{proposition}[Seb\H{o}~\cite{sebo1990}]  \label{prop:adjdist} 
Let $(G, T)$ be a bipartite graft, let $r \in V(G)$, and let $F \subseteq E(G)$ be a minimum join. 
Let $u$ and $v$ be adjacent vertices in $G$. 
Then, $|\distgtf{G}{T}{F}{r}{u} - \distgtf{G}{T}{F}{r}{v}| = 1$. 
\end{proposition}

\begin{definition} 
For a graft $(G, T)$ and $r\in V(G)$, 
we denote by $\interval{G}{T}{r}$ the set $\{  \distgt{G}{T}{r}{x} : x\in V(G) \}$ of integers. 
\end{definition}

\begin{definition} 
Let $(G, T)$ be a graft, and let $r\in V(G)$. 
For each $i\in \interval{G}{T}{r}$,  $\levelgtr{G}{T}{r}{i}$ denotes the set $\{ x\in V(G): \distgt{G}{T}{r}{x} = i \}$.   
The set $\{ x \in V(G) : \distgt{G}{T}{r}{x} \le  i\}$ is denoted by $\laylegtr{G}{T}{r}{i}$. 
The set of connected components of the subgraph of $G$ induced by $\laylegtr{G}{T}{r}{i}$ is denoted by $\laycompgtr{G}{T}{r}{i}$. 
The union of the sets $\laycompgtr{G}{T}{r}{i}$ over all  $i \in \interval{G}{T}{r}$ is denoted by $\laycompall{G}{T}{r}$. 
The members of $\laycompall{G}{T}{r}$ are called {\em distance components} of $(G, T)$ with respect to $r$. 
A distance component $L \in \laycompall{G}{T}{r}$ is said to be {\em capital} if $r \in V(L)$ holds; 
otherwise, $L$ is said to be {\em decapital}. 
The set of decapital members in $\laycompall{G}{T}{r}$ is denoted by $\noncapall{G}{T}{r}$. 
\end{definition}

\begin{definition} 
Let $(G, T)$ be a graft, and let $F$ be a minimum join. Let $r\in V(G)$. 
Graft $(G, T)$ is {\em primal} with respect to $r$ 
if $\distgtf{G}{T}{F}{r}{x} \le 0$ holds for every $x\in V(G)$. 
\end{definition}

Theorem~\ref{thm:sebo} presents a part of the main results from Seb\H{o}~\cite{sebo1990}.

\begin{theorem}[Seb\H{o}~\cite{sebo1990}]  \label{thm:sebo} 
Let $(G, T)$ be a bipartite graft, and let $F$ be a minimum join. Let $r\in V(G)$. 
Then, the following properties hold. 
\begin{rmenum} 
\item \label{item:sebo:cap} For every capital member $K \in \laycompall{G}{T}{r}$, $\parcut{G}{L}\cap F = \emptyset$. 
\item \label{item:sebo:noncap} For every decapital member $K \in \laycompall{G}{T}{r}$, $|\parcut{G}{L}\cap F| = 1$. 
\item \label{item:sebo:inpath} Let $K \in \noncapall{G}{T}{r}$, and let $r_K \in \shore{G}{\parcut{G}{K}\cap F} \cap V(K)$. 
Then $(G, T)_F[K]$ is a primal graft with root $r_K$, in which $F\cap E(K)$ is a minimum join. 
Furthermore, let $i_K := \distgtf{G}{T}{F}{r}{r_K}$ and $I := \{ i\in \interval{G}{T}{r}:  \levelgtr{G}{T}{r}{i} \cap V(K) \neq \emptyset\}$; 
then for every $i\in I$ and every $x\in \levelgtr{G}{T}{r}{i}\cap V(K)$, 
the $F$-distance in $(G, T)_F[K]$ between $r_K$ and $x$ is $i - i_K$. 
\end{rmenum} 
\end{theorem}

\begin{definition} 
Let $(G, T)$ be a bipartite graft, and let $r \in V(G)$. Let $F$ be a minimum join. 
Let $K \in \noncapall{G}{T}{r}$. 
Under Theorem~\ref{thm:sebo},  let $e$ be the edge in $F\cap \parcut{G}{K}$, 
and let $r_K \in V(K)$ and $s_K \in V(G)\setminus V(K)$ be the ends of $e$. 
We call $e$ the {\em $F$-beam} of $K$, 
and $r_K$ and $s_K$ be the {\em $F$-root} and {\em -antiroot} of $K$, respectively. 
\end{definition}

\begin{definition} 
Let $(G, T)$ be a bipartite graft, and let $F$ be a minimum join. Let $r\in V(G)$. 
For $K \in \laycompall{G}{T}{r}$, 
we denote the sets $V(K)\cap \levelgtr{G}{T}{r}{i}$ and $V(K)\setminus \levelgtr{G}{T}{r}{i}$ by $\ak{G}{T}{K}$ and $\dk{G}{T}{K}$, respectively,  
where $i \in \interval{G}{T}{r}$ is such that $K \in \laycompgtr{G}{T}{r}{i}$.  
We denote the set of connected components of $G[\dk{G}{T}{K}]$ by $\dkconn{G}{T}{K}$. 
\end{definition}

\begin{remark} 
Note that Theorem~\ref{thm:sebo}~\ref{item:sebo:inpath} implies the following. 
In the primal graft $(G, T)_F[K]$ with respect to root $r_K$, 
$\ak{G}{T}{K}$ is the set of vertices whose $F$-distance from $r_K$ is zero. 
\end{remark}

\section{Factor-Connectivity and Kotzig--Lov\'asz Decomposition for Grafts} 

In this section, we present the Kotzig--Lov\'asz decomposition for grafts~\cite{DBLP:journals/dam/Kita22}. 

\begin{definition} 
Let $(G, T)$ be a graft. 
We say that an edge $e\in E(G)$ is {\em allowed} if $(G, T)$ has a minimum join $F$ with $e \in F$. 
Two vertices $x, y\in V(G)$ are said to be {\em factor-connected} if $G$ has a path between $x$ and $y$ in which every edge is allowed in $(G, T)$. 
A {\em factor-connected component} or a {\em factor-component} of $(G, T)$ is a maximal subgraph of $G$ in which every two vertices are factor-connected in $(G, T)$. 
The set of factor-components of $(G, T)$ is denoted by $\tcomp{G}{T}$. 
\end{definition}

\begin{definition} 
Let $(G, T)$ be a graft. 
For $u, v\in V(G)$, 
we say that $u \gtsim{G}{T} v$ if $u = v$ or if $u$ and $v$ are distinct vertices from the same factor-component of $(G, T)$ 
such that $\nu(G, T) = \nu(G, T\Delta \{u, v\})$. 
\end{definition} 

\begin{theorem}[Seb\H{o}~\cite{sebo1990}, Kita~\cite{DBLP:journals/dam/Kita22}]  \label{thm:tkl} 
Let $(G, T)$ be a graft. 
Then, $\gtsim{G}{T}$ is an equivalence relation on $V(G)$. 
Every equivalence class of $\gtsim{G}{T}$ is contained in a factor-component of $(G, T)$. 
\end{theorem}

\begin{definition} 
Under Theorem~\ref{thm:tkl}, 
we call the family $V(G)/ \gtsim{G}{T}$ of equivalence classes 
 the {\em general Kotzig--Lov\'asz decomposition} or simply the {\em Kotzig--Lov\'asz decomposition} of $(G, T)$ 
and denote it by $\tpart{G}{T}$. 
For $C \in \tcomp{G}{T}$, the set of equivalence classes that are contained in $C$ 
is denoted by $\pargtpart{G}{T}{C}$. 
\end{definition}

\section{Basic Lemmas on Minimum Joins} 

In this section, we present lemmas on fundamental properties of minimum joins in grafts. 
See also Seb\H{o}~\cite{sebo1990} or Kita~\cite{DBLP:journals/dam/Kita22}.

\begin{lemma}  \label{lem:minjoin} 
Let $(G, T)$ be a graft. 
Then, $F\subseteq E(G)$ is a minimum join of $(G, T)$ if and only if, 
for every circuit $C$ in $G$, $w_F(C) \ge 0$. 
\end{lemma}

\begin{lemma} \label{lem:alt} 
Let $(G, T)$ be a graft and $F$ be a minimum join. 
If $C$ is a circuit with $w_F(C) = 0$, 
then $F\Delta E(C)$ is also a minimum join of $(G, T)$. 
\end{lemma}

Lemma~\ref{lem:allowed2dist} can easily be confirmed from Lemma~\ref{lem:minjoin}. 

\begin{lemma}  \label{lem:allowed2dist} 
Let $(G, T)$ be a graft and $F$ be a minimum join. 
Let $x$ and $y$ be adjacent vertices. 
Then, the edge $xy$ is allowed if and only if $\distgtf{G}{T}{F}{x}{y} = -1$. 
\end{lemma}

\section{Preliminary Results on Bipartite Grafts} 

In this section, we provide some preliminary results from Kita~\cite{DBLP:journals/dam/Kita22, kita2022tight} 
to be used for deriving the results in this paper.

\begin{definition} 
Let $(G, T)$ be a graft, and let $F$ be a minimum join. 
A set $X \subseteq V(G)$ is {\em extreme} if $\distgtf{G}{T}{F}{x}{y} \ge 0$ for every $x, y \in V(G)$. 
\end{definition}

\begin{lemma}[Kita~\cite{kita2022tight}] \label{lem:ak2extreme} 
Let $(G, T)$ be a bipartite graft, and let $r \in V(K)$. 
Let $K \in \laycompall{G}{K}{r}$, and let $i \in \interval{G}{T}{r}$ be such that $K \in \laycompgtr{G}{T}{r}{i}$. 
Then, $V(K) \cap \levelgtr{G}{T}{r}{i}$ is an extreme set in $(G, T)$. 
\end{lemma}

\begin{lemma}[Kita~\cite{kita2022tight}]  \label{lem:spineroot} 
Let $(G, T)$ be a primal bipartite graft with respect to $r \in V(G)$.   
Let $F$ be a minimum join of $(G, T)$. 
Then, for every $r' \in \agtr{G}{T}{r}$ and every $x \in V(G)$,  
$\distgtf{G}{T}{F}{r}{x} \le \distgtf{G}{T}{F}{r'}{x}$. 
\end{lemma}

\begin{theorem}[Kita~\cite{DBLP:journals/dam/Kita22, kita2022tight}]  \label{thm:distunit} 
Let $(G, T)$ be a bipartite graft, and let $r \in V(K)$.  
Let $F$ be a minimum join of $(G, T)$. 
Let $S \in \tpart{G}{T}$, and let $x_1, x_2 \in S$. 
Then, for every $y\in V(G)$, 
$\distgtf{G}{T}{F}{x_1}{y} = \distgtf{G}{T}{F}{x_2}{y}$. 
\end{theorem}

\begin{lemma}[Kita~\cite{kita2022tight}] \label{lem:ear2nonneg} 
Let $(G, T)$ be a graft, and let $F$ be a minimum join.  
Let $r \in V(G)$, and let $K \in \laycompall{G}{T}{r}$. 
Let $P$ be a path in $G$ whose ends are in $K$ and whose internal vertices are disjoint from $K$.
Then, $w_F(P) \ge 0$ holds.  
Moreover, $w_F(P) = 0$ holds only if $K$ is a decapital distance component,  
one end of $P$ is the $F$-root of $K$, and $P$ contains the $F$-beam of $K$.  
\end{lemma}

\begin{definition} 
Let $(G, T)$ be a graft, and let  $F$ be a minimum join.  
Let $S \subseteq V(G)$.  
A set $X \subseteq V(G)\setminus S$ is  $F$-{\em negative} for $S$ 
if, for every $x\in X$, there is a path between $x$ and a vertex in $S$  
such that its $F$-weight is negative and its vertices except the end in $S$ are contained in $X$.  
\end{definition} 

\begin{remark} 
It is easily observed that the maximum $F$-negative set exists for every $S$. 
\end{remark}

\begin{proposition}[Kita~\cite{kita2022tight}]  \label{prop:negset} 
Let $(G, T)$ be a graft, and let $S \subseteq V(G)$.  
If $F_1$ and $F_2$ are two minimum joins of $(G, T)$, then the maximum $F_1$-negative set for $S$ is equal to 
the maximum $F_2$-negative set for $S$. 
\end{proposition}

\begin{definition} 
Let $(G, T)$ be a graft, and let $F$ be a minimum join.  
Let $S \subseteq V(G)$.  
Under Proposition~\ref{prop:negset},  
we refer to the maximum $F$-negative set for $S$ simply as the maximum negative set for $S$ 
and denote it by $\negset{G}{T}{S}$. 
\end{definition}

\section{Universality of Decapital Distance Components}

In this section, we establish Theorem~\ref{thm:universal}.  
We first present Lemmas~\ref{lem:path2nonroot}, \ref{lem:path2kroot}, and~\ref{lem:path2sum}, which together yield Lemma~\ref{lem:distalt}.  
From Lemma~\ref{lem:distalt}, we derive Lemmas~\ref{lem:noncap2universal} and~\ref{lem:dijoin2noncap},  
which are then used to prove Theorem~\ref{thm:universal}.  
In Corollary~\ref{cor:universal}, we determine the total number of decapital distance components in a bipartite graft, taken over all root choices.

\begin{definition} 
Let $(G, T)$ be a bipartite graft, and let $r \in V(G)$. Let $F$ be a minimum join. 
Let $K \in \noncapall{G}{T}{r}$. 
Let $r_K$ and $s_K$ be the $F$-root and -antiroot of $K$, respectively. 
We say that a vertex $r' \in V(G)$  is {\em $K$-congruent} to $r$ if 
 $\distgtf{G}{T}{F}{r'}{r_K} = \distgtf{G}{T}{F}{r'}{s_K}  -1$. 
\end{definition}

\begin{lemma}  \label{lem:path2nonroot} 
Let $(G, T)$ be a bipartite graft and $F \subseteq E(G)$ be a minimum join. Let $r \in V(G)$. 
Let $K \in \noncapall{G}{T}{r}$. Let $e \in \parcut{G}{K}$ be the $F$-beam of $K$, 
and let $r_K$ and $s_K$ be the $F$-root and -antiroot of $K$, respectively. 
Let $r' \in V(G)\setminus V(K)$. 
Let $P$ be a path between $r'$ and a vertex $x \in V(K)$ such that $V(P)\cap V(K) = \{x\}$.  
Then $w_F(P) \ge \distgtf{G}{T}{F}{r'}{r_K}$.  
Furthermore, if $r'$ is $K$-congruent to $r$,  
then $w_F(P) = \distgtf{G}{T}{F}{r'}{r_K}$ holds only if $x = r_K$.   
\end{lemma} 
\begin{proof} The definition of $x$ implies $x\in\ak{G}{T}{K}$. 
By Theorem~\ref{thm:sebo} \ref{item:sebo:inpath}, there is a path $Q$ in $K$ between $r_K$ and $x$ such that $w_F(Q) = 0$.  
Let $R := P + Q$. 
Then $R$ is a path between $r'$ and $r_K$ with $w_F(R) = w_F(P)$
In addition, $w_F(R) \ge \distgtf{G}{T}{F}{r'}{r_K}$ clearly holds. 
Thus, $w_F(P) \ge \distgtf{G}{T}{F}{r'}{r_K}$.  This proves the first statement. 

To prove the second statement, let $\distgtf{G}{T}{F}{r'}{r_K} = \distgtf{G}{T}{F}{r'}{s_K} - 1$ 
and  $x \neq r_K$. 
First, consider the case $s_K \not\in V(P)$. 
Then $R + e$ is a path between $r'$ and $s_K$ with $w_F(R + e) = w_F(P) -1$.  
Additionally, with the assumption on $r'$,   $w_F(R + e) \ge \distgtf{G}{T}{F}{r'}{s_K} = \distgtf{G}{T}{F}{r'}{r_K} + 1$. 
It follows that $w_F(P) - 1  \ge \distgtf{G}{T}{F}{r'}{r_K} +1$. Thus, $w_F(P) > \distgtf{G}{T}{F}{r'}{r_K}$. 

Next, consider the case $s_K \in V(P)$. 
Then $s_KPx + Q + e$ is a circuit with $w_F(s_KPx + Q + e) = w_F(s_KPx) -1$; 
hence, we obtain $w_F(s_KPx) \ge 1$ from Lemma~\ref{lem:minjoin}.  
Accordingly, 
$w_F(P) = w_F(r'Ps_K) + w_F(s_KPx) \ge w_F(r'Ps_K) + 1$. 
Because $r'Ps_K + e$ is a path between $r'$ and $r_K$, 
we also have $w_F(r'Ps_K + e) = w_F(r'Ps_K) - 1 \ge \distgtf{G}{T}{F}{r'}{r_K}$.  
Therefore, we obtain $w_F(P) > \distgtf{G}{T}{F}{r'}{r_K}$.   
The lemma is proved. 
\end{proof}

Lemma~\ref{lem:path2nonroot}, together with Lemma~\ref{lem:ak2extreme}, implies Lemma~\ref{lem:path2kroot}.

\begin{lemma} \label{lem:path2kroot}  
Let $(G, T)$ be a bipartite graft and $F \subseteq E(G)$ be a minimum join. Let $r \in V(G)$. 
Let $K \in \noncapall{G}{T}{r}$. Let $e \in \parcut{G}{K}$ be the $F$-beam of $K$, 
and let $r_K$ and $s_K$ be the $F$-root and -antiroot of $K$, respectively. 
Let $r' \in V(G)\setminus V(K)$  be a vertex that is $K$-congruent to $r$. 
Then, every $F$-shortest path between $r'$ and $r_K$ contains $e$ and is disjoint from $V(K)$ except for the end $r_K$. 
\end{lemma} 
\begin{proof} Let $P$ be an $F$-shortest path between $r'$ and $r_K$. 
The assumption on $r'$ implies the following claim. 
\begin{pclaim}  \label{claim:last} 
$e\in E(P)$. 
\end{pclaim} 
\begin{proof}  If $s_K \not\in V(P)$ holds, then $P + e$ is a path between $r'$ and $s_K$ with smaller $F$-weight than $P$. 
This contradicts the assumption on $r'$. 
Thus, $s_K\in V(P)$ follows. 
Then, $r'Ps_K + e$ is a path between $r'$ and $r_K$ whose $F$-weight is $w_F(r'Ps_K) -1$. 
As $P$ is an $F$-shortest path between $r'$ and $r_K$, we have $w_F(s_KPr_K) \le -1$. 
If $e\not\in P$ holds, then $s_KPr_K + e$ is a circuit of negative $F$-weight. This contradicts Lemma~\ref{lem:minjoin}. 
Thus, $e\in E(P)$ is obtained. 
\end{proof} 

Suppose  $V(P)\cap V(K) \neq \{ r_K\}$.   
Trace $P$ from $r'$, and let $x$ be the first encountered vertex in $V(K)$. 
The supposition and Claim~\ref{claim:last} imply $s_K \not \in V(r'Px)$. 
Thus, Lemma~\ref{lem:path2nonroot} implies $w_F(r'Px) > \distgtf{G}{T}{F}{r'}{r_K}$. 
Because $xPr_K$ has both ends in $\ak{G}{T}{K}$, Lemma~\ref{lem:ak2extreme} implies  $w_F(xPr_K) \ge 0$. 
Therefore, $w_F(P) > \distgtf{G}{T}{F}{r'}{r_K}$, 
which contradicts the definition of $P$. 
Thus, $V(P)\cap V(K) = \{ r_K\}$ follows.    
The lemma is proved. 
\end{proof}

Lemmas~\ref{lem:path2nonroot} and~\ref{lem:path2kroot}, in combination with Lemmas~\ref{lem:ak2extreme}, \ref{lem:spineroot}, and~\ref{lem:ear2nonneg}, imply Lemma~\ref{lem:path2sum}.

\begin{lemma}  \label{lem:path2sum} 
Let $(G, T)$ be a bipartite graft, let $r \in V(G)$, and let $F \subseteq E(G)$ be a minimum join. 
Let $K \in \noncapall{G}{T}{r}$. Let $e \in \parcut{G}{K}$ be the $F$-beam of $K$, 
and let $r_K$ and $s_K$ be the $F$-root and -antiroot of $K$, respectively. 
Let $r' \in V(G)\setminus V(K)$  be a vertex that is $K$-congruent to $r$. 
Let $x\in V(K)$. 
Then a path between $r'$ and $x$ is $F$-shortest between $r'$ and $x$ if and only if 
it is the sum of an $F$-shortest path between $r'$ and $r_K$ 
and an $F$-shortest path in $(G, T)_F[K]$ between $r_K$ and $x$. 
\end{lemma} 
\begin{proof}  Let $P$ be an $F$-shortest path between $r'$ and $r_K$. 
Note Theorem~\ref{thm:sebo} \ref{item:sebo:inpath} in the following,  
and let $Q$ be an $F$-shortest path  in $(G, T)_F[K]$ between $r_K$ and $x$. 
Lemma~\ref{lem:path2kroot} implies that $P + Q$ is a path between $r'$ and $r_K$.

Let $R$ be an arbitrary path between $r'$ and $x$. 
Trace $R$ from $x$, and let $y$ be the first encountered vertex that is also connected to a vertex in $V(G)\setminus V(K)$. 
Trace $R$ from $r'$, and let $z$ be the first encountered vertex in $K$. 

From Lemma~\ref{lem:spineroot}, we have $w_F(xRy) \ge w_F(Q)$. 
We also have 
$w_F(yRz) \ge 0$ by Lemma~\ref{lem:ak2extreme} and $w_F(zRr') \ge w_F(P)$ by Lemma~\ref{lem:path2nonroot}.  
It follows that $w_F(R) = w_F(xRy) + w_F(yRz) + w_F(zRr')  \ge w_F(Q) + 0 + w_F(P) = w_F(P+Q)$, and 
we obtain $w_F(R) \ge w_F(P+Q)$. 
Because this inequality holds for every path $R$ between $r'$ and $x$, it follows that $P+Q$ is an $F$-shortest path between $r'$ and $x$; 
this proves the necessity of the statement. 
The equality holds if and only if $w_F(xRy) = w_F(Q)$, $w_F(yRz) = 0$, and $w_F(zRr') = w_F(P)$.

Assume that $R$ is an $F$-shortest path between $r'$ and $x$. 
Then $w_F(xRy) = w_F(Q)$, $w_F(yRz) = 0$, and $w_F(zRr') = w_F(P)$. 
From $w_F(zRr') = w_F(P)$, Lemmas~\ref{lem:path2nonroot} and \ref{lem:path2kroot} imply $z = r_K$ and $e \in \parcut{P}{r_K}$. 
Lemma~\ref{lem:ear2nonneg} further implies $y = z$. 
Thus, $y = r_K = z$ follows, 
and $R$ is accordingly the sum of a path between $r'$ and $r_K$ and a path in $K$ between $r_K$ and $x$. 
It further follows that  $R$ is the sum of an $F$-shortest path between $r'$ and $r_K$ and a shortest path in $(G, T)_F[K]$ between $r_K$ and $x$. 
The sufficiency of the statement is proved. 
This completes the proof. 
\end{proof}

Lemmas~\ref{lem:path2nonroot}, \ref{lem:path2kroot}, and \ref{lem:path2sum} imply Lemma~\ref{lem:distalt}.

\begin{lemma} \label{lem:distalt} 
Let $(G, T)$ be a bipartite graft and $F \subseteq E(G)$ be a minimum join. Let $r \in V(G)$. 
Let $K \in \noncapall{G}{T}{r}$. Let $e \in \parcut{G}{K}$ be the $F$-beam of $K$, 
and let $r_K$ and $s_K$ be the $F$-root and -antiroot of $K$, respectively. 
Let $r' \in V(G)\setminus V(K)$ be a vertex that is $K$-congruent to $r$. 
Let $d := \distgtf{G}{T}{F}{r'}{r_K} - \distgtf{G}{T}{F}{r}{r_K}$. 
Then, the following properties hold. 
\begin{rmenum} 
\item \label{item:distsum} For every $x\in V(K)$, $\distgtf{G}{T}{F}{r'}{x} = \distgtf{G}{T}{F}{r}{x} + d$. 
\item \label{item:neidist} For every $x \in \parNei{G}{K}$, $\distgtf{G}{T}{F}{r'}{x} = \distgtf{G}{T}{F}{r'}{r_K} + 1$. 
\end{rmenum} 
\end{lemma} 
\begin{proof} Lemma~\ref{lem:path2sum} easily implies \ref{item:distsum}. 
We prove \ref{item:neidist} in the following. 
For every $x \in \parNei{G}{K}$, and let $y \in \ak{G}{T}{K}$ be a vertex with $xy\in E(G)$. 
If $xy \in F$ holds, then Theorem~\ref{thm:sebo}~\ref{item:sebo:noncap} implies $x = s_K$ and $y = r_K$; the statement clearly holds. 
We assume $xy \not\in F$ in the following. 
Let $P$ be an $F$-shortest path between $r'$ and $x$. Let $i := \distgtf{G}{T}{F}{r}{r_K}$. 
As $y \in \ak{G}{T}{K}$ holds, Lemma~\ref{lem:path2sum} implies $\distgtf{G}{T}{F}{r'}{y} = i + d$. 
Proposition~\ref{prop:adjdist} implies $w_F(P) \in \{i + d + 1, i + d -1 \}$. 

First, consider the case $y\not\in V(P)$. 
Then, $P + xy$ is a path between $r'$ and $y$. 
Suppose $w_F(P) = i + d -1$. 
Then $w_F(P + xy) = i + d$; that is, $P + xy$ is an $F$-shortest path between $r'$ and $y$. 
If $y = r_K$, then this contradicts Lemma~\ref{lem:path2kroot}, because $xy \not\in F$ is assumed.  
If $y \neq r_K$, then it contradicts Lemma~\ref{lem:path2sum}.  
Hence, we obtain $w_F(P) = i + d + 1$ for this case.

Next, consider the case $y\in V(P)$. 
We have $w_F(P) = w_F(r'Py) + w_F(yPx)$.  
Lemma~\ref{lem:path2nonroot} implies $w_F(r'Py)  \ge i + d$.  
Proposition~\ref{prop:adjdist} implies $w_F(yPx) \ge -1$. 
\begin{pclaim} \label{claim:minimum} 
$w_F(r'Py) \neq i + d$ or $w_F(yPx) \neq -1$.  
\end{pclaim} 
\begin{proof} 
Suppose $w_F(r'Py) = i + d$ and $w_F(yPx) = -1$. 
Then $r'Py$ is an $F$-shortest path between $r'$ and $y$. 
Lemma~\ref{lem:path2kroot} implies $e\in E(r'Py)$. 
Consequently, $e \not\in E(yPx)$. 

Let $C:= yPx + xy$. Then,  $C$ is a circuit with $w_F(C) = 0$.  
Lemma~\ref{lem:alt} further implies that $F\Delta E(C)$ is a minimum join of $(G, T)$ with $|\parcut{G}{K}\cap (F\Delta E(C))| > 1$. 
This contradicts Theorem~\ref{thm:sebo} \ref{item:sebo:noncap}.  
Thus, the claim is proved. 
\end{proof} 

By Claim~\ref{claim:minimum}, we obtain $w_F(P) > i + d - 1$, and thus $w_F(P) = i + d + 1$  follows. 
The proof of the lemma is complete. 
\end{proof}

Lemma~\ref{lem:distalt} implies Lemma~\ref{lem:noncap2universal}.

\begin{lemma}  \label{lem:noncap2universal} 
Let $(G, T)$ be a bipartite graft, let $r \in V(G)$, and let $F \subseteq E(G)$ be a minimum join. 
Let $K \in \noncapall{G}{T}{r}$. Let $e \in \parcut{G}{K}$ be the $F$-beam of $K$, 
and let $r_K$ and $s_K$ be the $F$-root and -antiroot of $K$, respectively. 
Let $r' \in V(G)\setminus V(K)$ be a vertex that is $K$-congruent to $r$. 
Then, $K$ is a member of  $\noncapall{G}{T}{r'}$. 
\end{lemma} 
\begin{proof} Let $i := \distgtf{G}{T}{F}{r}{r_K}$ and let $d := \distgtf{G}{T}{F}{r'}{r_K} - \distgtf{G}{T}{F}{r}{r_K}$. 
Lemma~\ref{lem:distalt} implies  $V(K)\subseteq \laylegtr{G}{T}{r'}{i + d}$ and $\parNei{G}{K}\cap \laylegtr{G}{T}{r'}{i + d} = \emptyset$. 
Therefore, 
$K$ is also a connected component of $G[\laylegtr{G}{T}{r'}{i + d}]$.  
This proves the lemma. 
\end{proof}

Lemma~\ref{lem:dijoin2noncap} can be derived rather easily from Lemma~\ref{lem:noncap2universal}.

\begin{lemma}  \label{lem:dijoin2noncap} 
Let $(G, T)$ be a bipartite graft and $F \subseteq E(G)$ be a minimum join.  
\begin{rmenum} 
\item  \label{item:dijoin2noncap:antiroot} 
Let $r\in V(G)$  and  $K \in \noncapall{G}{T}{r}$. 
Let $r_K$ and $s_K$ be the $F$-root and -antiroot of $K$, respectively. 
Then, $K \in \laycompgtr{G}{T}{s_K}{-1}$ holds.  
\item \label{item:beam2noncap} 
For every $xy \in F$, there exists $K \in \laycompgtr{G}{T}{y}{-1}$ with $x \in V(K)$ and $y\in V(G)\setminus V(K)$. 
\end{rmenum} 
\end{lemma} 
\begin{proof} Proposition~\ref{prop:adjdist} implies that $r_K$ is contained in a member $K'$ of $\laycompgtr{G}{T}{s_K}{-1}$. 
Lemma~\ref{lem:noncap2universal} implies $K = K'$. 
Thus, \ref{item:dijoin2noncap:antiroot} follows. 
The statement \ref{item:beam2noncap} easily follows from \ref{item:dijoin2noncap:antiroot}. 
\end{proof}

Theorem~\ref{thm:universal} can now be obtained from Lemma~\ref{lem:noncap2universal} and Lemma~\ref{lem:dijoin2noncap}~\ref{item:beam2noncap}.

\begin{theorem}  \label{thm:universal}  
Let $(G, T)$ be a bipartite graft, and let $F \subseteq E(G)$ be a minimum join. 
For every $xy\in F$, 
there uniquely exists $K \in \bigcup_{r\in V(G)} \noncapall{G}{T}{r}$ 
such that $x$ and $y$ are the $F$-root and -antiroot of $K$, respectively. 
\end{theorem} 
\begin{proof} Lemma~\ref{lem:dijoin2noncap} \ref{item:beam2noncap} implies 
that there exists $K \in \noncapgtr{G}{T}{y}{-1}$ such that $x$ and $y$ are the $F$-root and -antiroot of $K$, respectively. 
For proving the uniqueness, let  $K' \in \noncapall{G}{T}{r'}$, where $r' \in V(G)$, be a decapital distance component for which $x$ and $y$ are the $F$-root and -antiroot, respectively. 
Proposition~\ref{prop:adjdist} implies $\distgtf{G}{T}{F}{r'}{x} = \distgtf{G}{T}{F}{r'}{y} - 1$. 
Then Lemma~\ref{lem:noncap2universal} implies $K' = K$. 
This completes the proof. 
\end{proof}

\begin{corollary}  \label{cor:universal} 
Let $(G, T)$ be a bipartite graft, and let $r\in V(G)$. 
Then, $|\noncapall{G}{T}{r}| = \nu(G, T)$. 
\end{corollary} 
\begin{proof} Let $\hat{F} := \{ (u, v), (v, u) : uv \in F\}$.  
Under Lemma~\ref{lem:dijoin2noncap}~\ref{item:beam2noncap}, for each $(u, v) \in \hat{F}$, let $K_{ (u, v) }$ be a member of  $\bigcup_{r\in V(G)} \noncapall{G}{T}{r}$ with $u\in V(K)$ and $v\not\in V(K)$.  
Let $f: \hat{F}\rightarrow \bigcup_{r\in V(G)} \noncapall{G}{T}{r}$ 
be the mapping such that $f: (u, v) \mapsto K_{(u, v)}$. 
Theorem~\ref{thm:sebo} \ref{item:sebo:noncap} implies that $f$ is a surjection.  
Theorem~\ref{thm:universal} implies that $f$ is an injection. 
Thus, $f$ is a bijection between $\hat{F}$ and $\bigcup_{r\in V(G)} \noncapall{G}{T}{r}$. 
This proves the second statement. 
\end{proof}

\begin{remark} 
It is easily confirmed that 
$|\bigcup_{r\in V(G)} \noncapall{G}{T}{r}| = 2\nu(G, T)$. 
\end{remark}

\section{Root and Antiroot Classes of Decapital Distance Components}

In this section, we provide Lemma~\ref{lem:in2equiv} and use this lemma to obtain Theorem~\ref{thm:shore2equiv}. 
We then introduce the concepts of root and antiroot classes, and root fragments based on  Theorem~\ref{thm:shore2equiv}.

\begin{lemma} \label{lem:in2equiv} 
Let $(G, T)$ be a bipartite graft, let $r \in V(G)$, and let $F \subseteq E(G)$ be a minimum join. 
Let $K\in \noncapall{G}{T}{r}$. 
Then, all allowed edges in $\parcut{G}{K}$ are contained in the same factor-component of $(G, T)$.  
\end{lemma} 
\begin{proof} Let $e \in \parcut{G}{K}\cap F$, and let $f \in \parcut{G}{K}\setminus F$ be an allowed edge. 
Let $x$ and $y$ be the ends of $e$ and $f$, respectively, that are contained in $V(K)$. 
Theorem~\ref{thm:sebo} \ref{item:sebo:inpath} states that $K$ has a path $P$ between $x$ and $y$ with $w_F(P) = 0$. 
Let $F'$ be a minimum join of $(G, T)$ with $f\in F$. 
Then $(F' \setminus E(K) ) \dot\cup ( (F\cap E(K)) \Delta E(P) )$ is also a minimum join of $(G, T)$. 
Therefore, every edge of $P$ is allowed. 
Accordingly, $x$ and $y$ are factor-connected. 
This proves the lemma. 
\end{proof}

\begin{theorem}  \label{thm:shore2equiv} 
Let $(G, T)$ be a bipartite graft, and let $\tilde{F}$ be the set of allowed edges of $(G, T)$. 
Let $r \in V(G)$ and $K\in \noncapall{G}{T}{r}$. 
Then, there are a factor-component $C \in \tcomp{G}{T}$ and two distinct equivalence classes $S_1, S_2 \in \pargtpart{G}{T}{C}$ 
such that $\shore{G}{\tilde{F} \cap \parcut{G}{K} } \cap V(K) \subseteq S_1$ and $\shore{G}{\tilde{F} \cap \parcut{G}{K} } \setminus V(K) \subseteq S_2$. 
\end{theorem} 
\begin{proof} Lemma~\ref{lem:in2equiv} implies that 
the vertices in $\shore{G}{\tilde{F} \cap \parcut{G}{K} } \cap V(K)$ are factor-connected; 
Lemma~\ref{lem:ak2extreme} further implies that they are contained in a member of $\pargtpart{G}{T}{C}$. 
Similarly, Lemma~\ref{lem:in2equiv} also implies that the vertices in $\shore{G}{\tilde{F} \cap \parcut{G}{K} } \setminus V(K)$ are factor-connected. 
Let $i \in \interval{G}{T}{r}$ be such that $K \in \laycompgtr{G}{T}{r}{i}$ and $L$ be the member of $\laycompgtr{G}{T}{r}{i+1}$ with $V(K)\subseteq V(L)$.  
Proposition~\ref{prop:adjdist} implies that $\parNei{G}{K}$ is contained in $V(L)\cap \levelgtr{G}{T}{r}{i+1}$. 
Hence, Lemma~\ref{lem:ak2extreme} implies that $\shore{G}{\tilde{F} \cap \parcut{G}{K} } \setminus V(K)$ is contained in a member of $\pargtpart{G}{T}{C}$. 
Thus, the theorem is proved. 
\end{proof}

\begin{definition} 
Let $(G, T)$ be a bipartite graft, and let $r \in V(G)$. 
Under Theorem~\ref{thm:shore2equiv}, 
for each $K \in \noncapall{G}{T}{r}$, we call the member $S$ of $\tpart{G}{T}$  
the {\em root class} (resp. {\em antiroot  class}) of $K$ if  
$S \subseteq V(K)$ (resp. $S \cap V(K) = \emptyset$) holds and 
$S$ contains the ends of allowed edges from $\parcut{G}{K}$. 
For the root class $S_K$ of $K$, we call the set $V(K)\cap S_K$ the {\em root fragment} of $K$. 
\end{definition}

\section{Structure of Decapital Distance Components}

\subsection{Canonical Structure of Decapital Distance Components} 

In this section, we establish Theorem~\ref{thm:ak2part}, 
which describes the internal structure of a decapital distance component $K$ in a bipartite graft $(G, T)$ 
in terms of the members of $\tpart{G}{T}$ that constitute $\ak{G}{T}{K}$. 
The analogue of this theorem for capital distance components is provided in Kita~\cite{kita2022tight}. 

\begin{definition} 
Let $(G, T)$ be a bipartite graft, and let $r \in V(G)$. 
Let $K \in \noncapall{G}{T}{r}$. 
Let $S \subseteq V(G)$ be a set of vertices such that $S\cap \ak{G}{T}{K} \neq \emptyset$.  
We denote by $\neicompk{G}{T}{K}{S}$ the family of members from $\dkconn{G}{T}{K}$ that are adjacent to a vertex in $S$ with an allowed edge. 
The set of vertices contained in the members of $\neicompk{G}{T}{K}{S}$ is denoted by $\neisetk{G}{T}{K}{S}$. 
\end{definition}

\begin{definition} 
Let $(G, T)$ be a bipartite graft, and let $r \in V(G)$. 
Let $K \in \noncapall{G}{T}{r}$.  
Let $S_K$ and $\tilde{S}_K$  be the root class and the root fragment of $K$, respectively. 
Let $\mathcal{S}$ be the family of members from $\tpart{G}{T}\setminus \{S_K\}$ that contain a vertex in $\ak{G}{T}{K}$. 
We denote the family $\{\tilde{S}_K \} \cup \mathcal{S}$ by $\akfam{G}{T}{K}$. 
\end{definition}

\begin{theorem}  \label{thm:ak2part} 
Let $(G, T)$ be a bipartite graft, and let $r \in V(G)$. 
Let $K \in \noncapall{G}{T}{r}$.   
Let $\tilde{S}_K$ be the root fragment of $K$. 
For each $S \in \akfam{G}{T}{K} \setminus \{ \tilde{S}_K \}$, let $C_S \in \tcomp{G}{T}$ be the factor-component with $S \in \pargtpart{G}{T}{C_S}$. 
Then, the following properties hold. 
\begin{rmenum} 
\item \label{item:ak2part:part}  $\akfam{G}{T}{K}$ is a partition of $\ak{G}{T}{K}$. 
\item  \label{item:ak2part:nonallowed} $V(C_S) \subseteq V(K)$ holds for every $S \in \akfam{G}{T}{K} \setminus \{ \tilde{S}_K \}$. 
\item  \label{item:ak2part:comp} $\dkconn{G}{T}{K} = \dot\bigcup \{ \neicompk{G}{T}{K}{S} : S \in \akfam{G}{T}{K} \}$. 
\end{rmenum} 
\end{theorem} 
\begin{proof}  Theorem~\ref{thm:shore2equiv} implies 
that, for every allowed edge in $\parcut{G}{K}$, its end in $K$ is contained in $\tilde{S}_K$.  
Additionally, Lemma~\ref{lem:ak2extreme} implies that the vertices in $\ak{G}{T}{K}\setminus \{\tilde{S}_K \}$ 
are not factor-connected with the vertices in $\tilde{S}_K$.   
Therefore, no vertex in $\ak{G}{T}{K}\setminus \{\tilde{S}_K \}$ can be factor-connected with vertices in $V(G)\setminus V(K)$. 
Thus, \ref{item:ak2part:nonallowed} follows. 
Consequently, $S \subseteq  V(K)$ follows for every $S\in \akfam{G}{T}{K} \setminus \{\tilde{S}_K \}$. 
Theorem~\ref{thm:distunit} further implies $S \subseteq \ak{G}{T}{K}$. 
Thus, \ref{item:ak2part:part} follows. 

Next, we prove \ref{item:ak2part:comp}. 
Theorem~\ref{thm:sebo}~\ref{item:sebo:noncap} implies that every $L \in \dkconn{G}{T}{K}$ has a member $S$ of $\akfam{G}{T}{K}$ with  $L \in \neicompk{G}{T}{S}{K}$.  
As Lemma~\ref{lem:ak2extreme} implies that  no distinct members from $\akfam{G}{T}{K}$ can be contained in the same factor-component, 
Lemma~\ref{lem:in2equiv} further implies such $S$ uniquely exists for $L$. 
This proves \ref{item:ak2part:comp}. 
The theorem is proved. 
\end{proof}

\subsection{Negative Sets Avoiding Specified Vertex Sets}

We introduce the concept of $F$-negative sets that avoid a specified set of vertices.

\begin{definition} 
Let $(G, T)$ be a graft, and let  $F$ be a minimum join.  
Let $S \subseteq V(G)$.  
For $Y \in V(G)\setminus S$, 
we say that an $F$-negative set $X$ {\em avoids} $Y$  if $X\cap Y = \emptyset$. 
\end{definition}

\begin{observation} \label{obs:neg2max} 
Let $(G, T)$ be a graft, and let  $F$ be a minimum join.  
Let $S \subseteq V(G)$. Let $X_1, X_2 \subseteq V(G)\setminus S$. 
If $X_1$ and $X_2$ are $F$-negative sets for $S$, 
then $X_1 \cup X_2$ is also a $F$-negative set for $S$.  
  Let $Y \subseteq V(G)\setminus S$. 
If $X_1$ and $X_2$ are $F$-negative sets for $S$ avoiding $Y$, 
then $X_1 \cup X_2$ is also an $F$-negative set for $S$ avoiding $Y$. 
\end{observation}

\begin{definition} 
Observation~\ref{obs:neg2max} implies that there is  the maximum $F$-negative set for $S$ avoiding $Y$. 
Under Observation~\ref{obs:neg2max}, 
we denote by $\negsetnf{G}{T}{F}{S}{Y}$ the maximum $F$-negative set for $S$ avoiding $Y$. 
\end{definition}

\subsection{Negative Sets and Decapital Distance Components}

In this section, we provide Lemmas~\ref{lem:neicomp2neg}, \ref{lem:neineicomp2nonneg}, \ref{lem:nei2posi}, and \ref{lem:ak2negset}, 
and use them to derive Theorem~\ref{thm:neicomp2negset}. 
This theorem characterizes $\neisetk{G}{T}{K}{S}$ in terms of the maximum negative set 
for a bipartite graft $(G, T)$, a decapital distance component $K$, and $S \in \akfam{G}{T}{K}$, 
and thus provides a refinement of Theorem~\ref{thm:ak2part}. 
This theorem is also an analogue of a result given in Kita~\cite{kita2022tight}.
Lemma~\ref{lem:ak2negset} can be derived from Theorem~\ref{thm:sebo}.

\begin{lemma}  \label{lem:neicomp2neg} 
Let $(G, T)$ be a bipartite graft, and let $r \in V(G)$. Let $F$ be a minimum join. 
Let $K \in \noncapall{G}{T}{r}$.  
Let $S_K \in \tpart{G}{T}$ be the root class and $T_K \in \tpart{G}{T}$ be the antiroot class of $K$. 
Let $L \in \dkconn{G}{T}{K}$. 
Let $T_L \in \tpart{G}{T}$ be the antiroot class of $L$, and let $\tilde{T}_L := T_L \cap \ak{G}{T}{K}$.  
Then, the following properties hold. 
\begin{rmenum} 
\item \label{item:neicomp2neg:path} For every $x \in V(L)$, there is a path of negative $F$-weight between $x$ and a vertex $y \in \tilde{T}_L$ whose vertices except $y$ are contained in $V(L)$.  
\item \label{item:neicomp2neg:set} Accordingly, $V(L)$ is contained in the maximum negative set for $\tilde{T}_L$ avoiding $T_{K}$. 
\end{rmenum}  
\end{lemma} 
\begin{proof}  Theorem~\ref{thm:sebo} \ref{item:sebo:inpath} states that $L$ has a path $P$ between $x$ and the $F$-root of $L$ whose $F$-weight is no greater than $0$.  
Thus, $P + e$ is a path of negative $F$-weight between $x$ and $s_L$ 
whose vertices except $s_L$ are contained in $V(L)$, where $e$ and $s_L$ are the $F$-beam and $F$-antiroot of $L$.  
Clearly, $s_L \in \tilde{T}_L$ holds. 
Thus,  $P + e$ establishes \ref{item:neicomp2neg:path}. 
Also, $P + e$ is clearly disjoint from $T_K$. 
Thus, \ref{item:neicomp2neg:path} proves \ref{item:neicomp2neg:set}. 
This completes the proof of the lemma. 
\end{proof}

Lemma~\ref{lem:neineicomp2nonneg} can be obtained by applying Lemma~\ref{lem:ak2extreme}.

\begin{lemma}  \label{lem:neineicomp2nonneg} 
Let $(G, T)$ be a bipartite graft, and let $r \in V(G)$. Let $F$ be a minimum join. 
Let $K \in \noncapall{G}{T}{r}$. 
$L \in \dkconn{G}{T}{K}$, and let $T_L \in \tpart{G}{T}$ be the antiroot class of $L$. 
Let $S \in \akfam{G}{T}{K} \setminus \{ T_L \cap \ak{G}{T}{K} \}$.  
Then, the following properties hold. 
\begin{rmenum} 
\item \label{item:neineicomp2nonneg:ak} 
No path between a vertex in $\ak{G}{T}{L}$ and  a vertex in $S$ can be of negative $F$-weight if it is disjoint from $\ak{G}{T}{K}\setminus S$.  
\item \label{item:neineicomp2nonneg:nonneg} Consequently, $V(L)$ is disjoint from the maximum negative set for $S$.  
\end{rmenum} 
\end{lemma} 
\begin{proof} Let  $P$ be a path between $x \in \ak{G}{T}{L}$ and $s \in S$, 
and assume that $P$ does not contain any vertex in $\ak{G}{T}{K} \setminus S$.  
Trace $P$ from $x$, and let $y$ be the first encountered vertex in $S$. 
Then, $xPy$ is the sum of a path $Q$ of $L$ between vertices in $\ak{G}{T}{L}$ and an edge $e \in E_G[L, S]$. 
Lemma~\ref{lem:ak2extreme} implies $w_F(Q) \ge 0$, and Lemma~\ref{lem:in2equiv} implies  $e \not\in F$.    
Additionally, because $yPs$ is a path between vertices in $S$, we have $w_F(yPs) \ge 0$.  
Hence, $w_F(P) = w_F(Q) + w_F(e) + w_F(yPs) > 0$ follows.  
This proves \ref{item:neineicomp2nonneg:ak}.

Lemma~\ref{lem:ak2extreme} implies that $\ak{G}{T}{K}\setminus S$ is disjoint from any negative set for $S$. 
Therefore, \ref{item:neineicomp2nonneg:ak} implies that $\ak{G}{T}{L}$ is disjoint from any negative set for $S$. 
This further implies \ref{item:neineicomp2nonneg:nonneg} 
because every path between a vertex in $L$ and a vertex in $S$ must contain a vertex in $\ak{G}{T}{L}$.  
The lemma is proved. 
\end{proof}

Lemma~\ref{lem:nei2posi} follows from Proposition~\ref{prop:adjdist}, Lemma~\ref{lem:allowed2dist}, and Theorem~\ref{thm:distunit}.

\begin{lemma}  \label{lem:nei2posi} 
Let $(G, T)$ be a bipartite graft, and let $r \in V(G)$. Let $F$ be a minimum join. 
Let $K \in \noncapall{G}{T}{r}$.  
Let $S_K$ and $T_K$ be the root and antiroot classes of $K$, respectively, and $\tilde{S}_K$ be the root fragment of $K$. 
Let $S \in \akfam{G}{T}{K}$.  
Then, the following properties hold. 
\begin{rmenum} 
\item \label{item:nei2posi:nei2posi} 
If $S \neq \tilde{S}_K$, then $\distgtf{G}{T}{F}{s}{x} = 1$ for every $x \in \parNei{G}{S} \setminus V(K)$ and every $s\in S$.  
If $S = \tilde{S}_K$, then $\distgtf{G}{T}{F}{s}{x} = 1$ for every $x \in \parNei{G}{S} \setminus V(K) \setminus T_K$ and every $s\in S$.  
\item \label{item:nei2posi:nei2disjoint} 
Accordingly, 
if $S \neq \tilde{S}_K$, then $V(G)\setminus V(K)$ is disjoint from the maximum negative set for $S$. 
If $S = \tilde{S}_K$, then $V(G)\setminus V(K)$ is disjoint from the maximum $F$-negative set for $S$ avoiding $T_K$. 
\end{rmenum} 
\end{lemma} 
\begin{proof} If $S \neq \tilde{S}_K$, let  $x \in \parNei{G}{S} \setminus V(K)$; 
otherwise, let $x \in \parNei{G}{S} \setminus V(K) \setminus T_K$. 
Let $y \in V(K)$ be a vertex with $xy\in E(G)$. 
Theorem~\ref{thm:shore2equiv} implies that $xy$ is not allowed in $(G, T)$. 
Therefore, Proposition~\ref{prop:adjdist} and Lemma~\ref{lem:allowed2dist} imply that $\distgtf{G}{T}{F}{x}{y} = 1$. 
The statement \ref{item:nei2posi:nei2posi} now follows from Theorem~\ref{thm:distunit}. 

Let $P$ be an arbitrary path between a vertex in $V(G)\setminus V(K)$ and a vertex in $S$.  
Then $P$ contains either a vertex in $\parNei{G}{S} \setminus V(K)$ 
or a vertex in $\ak{G}{T}{K} \setminus S$. 
However, if $S \neq \tilde{S}_K$, 
\ref{item:nei2posi:nei2posi} and Lemma~\ref{lem:ak2extreme} imply that 
there is no path of negative $F$-weight between any vertex in $\parNei{G}{S} \setminus V(K)$ or $\ak{G}{T}{K} \setminus S$ and any vertex in $S$; 
accordingly, $P$ cannot be contained in the maximum negative set for $S$. 
In contrast, if $S = \tilde{S}_K$, 
\ref{item:nei2posi:nei2posi} and Lemma~\ref{lem:ak2extreme} imply that 
there is no path of negative $F$-weight between any vertex in $\parNei{G}{S} \setminus V(K) \setminus T_K$ or $\ak{G}{T}{K} \setminus S$ and any vertex in $S$; 
accordingly, $P$ cannot be contained in the maximum negative set for $S$ if $P$ is disjoint from $T_K$. 
Consequently, \ref{item:nei2posi:nei2disjoint} follows. 
This completes the proof of the lemma. 
\end{proof}

Lemma~\ref{lem:ak2negset} can be obtained by combining Lemmas~\ref{lem:neicomp2neg}, \ref{lem:neineicomp2nonneg}, and \ref{lem:nei2posi}, together with Lemma~\ref{lem:ak2extreme} and Theorem~\ref{thm:ak2part}.

\begin{lemma}  \label{lem:ak2negset} 
Let $(G, T)$ be a bipartite graft, and let $r \in V(G)$. Let $F$ be a minimum join of $(G, T)$. 
Let $K \in \noncapall{G}{T}{r}$.   
Let $T_K$ be the antiroot class of $K$, and let $\tilde{S}_K$ be the root fragment of $K$. 
Then, the following properties hold. 
\begin{rmenum} 
\item \label{item:ak2negset:nei2neg} For each $S \in \akfam{G}{T}{K}$, the set $\neisetk{G}{T}{K}{S}$ is equal to $\negsetnf{G}{T}{F}{S}{T_K}$. 
\item \label{item:ak2negset:nonroot} Furthermore, if $S \neq \tilde{S}_K$, then $\negsetnf{G}{T}{F}{S}{T_K} = \negset{G}{T}{S}$. 
\item \label{item:ak2negset:part} Accordingly, $\dk{G}{T}{K} = \dot\bigcup_{ S \in \akfam{G}{T}{K} } \negsetnf{G}{T}{F}{S}{T_K}$. 
\end{rmenum} 
\end{lemma} 
\begin{proof} Let $S \in \akfam{G}{T}{K}$. 
Lemma~\ref{lem:neicomp2neg} implies $\neisetk{G}{T}{K}{S} \subseteq \negsetnf{G}{T}{F}{S}{T_K}$.  
Since $\negsetnf{G}{T}{F}{S}{T_K} \subseteq \negset{G}{T}{S}$  follows directly from the definition, 
$\neisetk{G}{T}{K}{S} \subseteq \negset{G}{T}{S}$ also holds.  
In the following, we prove $\neisetk{G}{T}{K}{S} \supseteq \negsetnf{G}{T}{F}{S}{T_K}$ 
by proving $V(G) \setminus \neisetk{G}{T}{K}{S} \subseteq V(G)\setminus \negsetnf{G}{T}{F}{S}{T_K}$.  
For the case $S \neq \tilde{S}_K$, 
we shall also prove $\neisetk{G}{T}{K}{S} \supseteq \negset{G}{T}{S}$  
by proving $V(G) \setminus \neisetk{G}{T}{K}{S} \subseteq V(G)\setminus \negset{G}{T}{S}$.   
Note $V(G) \setminus \neisetk{G}{T}{K}{S} = (V(K)\setminus \neisetk{G}{T}{K}{S}) \cup \ak{G}{T}{K} \cup (V(G)\setminus V(K))$.

According to Theorem~\ref{thm:ak2part} \ref{item:ak2part:comp} and Lemma~\ref{lem:neineicomp2nonneg}, 
$\dk{G}{T}{K} \setminus \neisetk{G}{T}{K}{S}$ is disjoint from $\negset{G}{T}{S}$.  
Lemma~\ref{lem:ak2extreme} implies that $\ak{G}{T}{K}$ is disjoint from $\negset{G}{T}{S}$.  
Thus, we have $V(K)\setminus \neisetk{G}{T}{K}{S} \subseteq V(G)\setminus \negset{G}{T}{S} \subseteq V(G)\setminus \negsetnf{G}{T}{F}{S}{T_K}$.  
Additionally, Lemma~\ref{lem:nei2posi} implies that $V(G)\setminus V(K)$ is disjoint from $\negsetnf{G}{T}{F}{S}{T_K}$; 
if $S \neq \tilde{S}_K$, then  $V(G)\setminus V(K)$ is also disjoint from $\negset{G}{T}{S}$.  

Consequently, we obtain $V(G) \setminus \neisetk{G}{T}{K}{S} \subseteq V(G)\setminus \negsetnf{G}{T}{F}{S}{T_K}$; 
thus, $\neisetk{G}{T}{K}{S} \supseteq \negsetnf{G}{T}{F}{S}{T_K}$ is proved, 
and $\neisetk{G}{T}{K}{S} = \negsetnf{G}{T}{F}{S}{T_K}$ follows.   
This proves \ref{item:ak2negset:nei2neg}. 
If $S \neq \tilde{S}_K$, then $\neisetk{G}{T}{K}{S} = \negset{G}{T}{S}$ can also be proved; 
thus, we  obtain $V(G) \setminus \neisetk{G}{T}{K}{S} \subseteq V(G)\setminus \negset{G}{T}{S}$. 
This proves \ref{item:ak2negset:nonroot}. 

The statement \ref{item:ak2negset:part} follows from Theorem~\ref{thm:ak2part}. 
The lemma is proved. 
\end{proof}

\begin{definition} 
Let $(G, T)$ be a bipartite graft, and let $r \in V(G)$. Let $F$ be a minimum join. 
Let $K \in \noncapall{G}{T}{r}$.  
Let $T_K$ be the antiroot class of $K$. 
Let $S \in \akfam{G}{T}{K}$. 
According to Lemma~\ref{lem:ak2negset}, 
the maximum $F$-negative set for $S$ avoiding $T_K$ does not depend on the choice of the minimum join $F$. 
As such, 
we simplify the notation and denote $\negsetnf{G}{T}{F}{S}{T_K}$ as $\negsetn{G}{T}{S}{T_K}$ 
for $S$.  
\end{definition}

We present Theorem~\ref{thm:neicomp2negset}, which combines and reformulates the statements of Theorem~\ref{thm:ak2part}~\ref{item:ak2part:part} and Lemma~\ref{lem:ak2negset}.

\begin{theorem}  \label{thm:neicomp2negset} 
Let $(G, T)$ be a bipartite graft, and let $r \in V(G)$. 
Let $K \in \noncapall{G}{T}{r}$.  
Let $\tilde{S}_K$ and $T_K$ be the root fragment and antiroot class of $K$, respectively. 
Then, the following properties hold. 
\begin{rmenum} 
\item  $\dkconn{G}{T}{K} = \dot\bigcup \{ \neicompk{G}{T}{K}{S} : S \in \akfam{G}{T}{K} \}$. 
Accordingly, $\dk{G}{T}{K} = \dot\bigcup \{ \neisetk{G}{T}{K}{S} : S \in \akfam{G}{T}{K}  \}$. 
\item $\neisetk{G}{T}{K}{\tilde{S}_K} = \negsetn{G}{T}{\tilde{S}_K}{T_K}$.  
For every $S \in \akfam{G}{T}{K} \setminus \{ \tilde{S}_K \}$, $\neisetk{G}{T}{K}{S} = \negsetn{G}{T}{S}{T_K} = \negset{G}{T}{S}$. 
Accordingly, 
$\dk{G}{T}{K}$ is equal to  $\dot\bigcup \{ \negsetn{G}{T}{S}{T_K} : S \in \akfam{G}{T}{K} \}$, 
which is further equal to $\negsetn{G}{T}{\tilde{S}_K}{T_K} \dot\cup \dot\bigcup \{ \negset{G}{T}{S} :  S \in \akfam{G}{T}{K} \setminus \{ \tilde{S}_K \} \}$. 
\end{rmenum} 
\end{theorem}

\begin{ac} 
This work was supported by JSPS KAKENHI Grant Numbers 18K13451 and 23K03192. 
\end{ac}

\bibliographystyle{splncs03.bst}
\bibliography{tdecap.bib}

\end{document}